\documentclass{amsart}

\usepackage{latexsym,amssymb}
\usepackage[english]{babel}

\usepackage{amsfonts, amsthm}
\usepackage{amsmath}

\newtheorem{prop}{Proposition}[section]

\newtheorem{thm}[prop]{Theorem}

\theoremstyle{definition}

\newtheorem*{KN}{The Knight's Tour Problem}

\newtheorem*{PN}{The Prince's Tour Problem}

\newtheorem*{OP}{Open Problem}

\begin{document}

\title[A solution to  the 7-prince's tour problem]{A solution to  the 7-prince's tour problem}

\author[L. Mella]{Lorenzo Mella}
\address{Dip. di Scienze Fisiche, Informatiche, Matematiche, Universit\`a degli Studi di Modena e Reggio Emilia, Via Campi 213/A, I-41125 Modena, Italy}
\email{lorenzo.mella@unipr.it}


\maketitle

\begin{abstract}
In \cite{CZ} the authors studied the closed tour problem on the $8\times 8$ chessboard of a chess piece, called $k$-prince, leaving open the existence of such a tour when $k=7$. In this note we find a solution to this open case.
\end{abstract}

\section{On the  prince's tour problem on a square chessboard}
Chessboard related problems have always been a widely studied topic in Discrete Mathematics, and often they arise from recreational purposes, see for instance \cite{W}. In particular, one of the most famous problems in this area is the following:
 \begin{KN}
Can a knight make a closed tour on an $8 \times 8$ chessboard visiting each cell exactly once?
\end{KN}
The affirmative solution to this problem has been known for centuries, and from there researchers considered knight's tours over chessboards of different shapes, or tours of different chess pieces. 

In \cite{CZ} the authors, starting from an irregular domination problem over grid graphs, proposed the movement of a \textit{prince} over a chessboard. Given a positive integer $k$, a \textit{$k$-prince} (in short $k$P) can move $k$ squares away from its current position: namely, if the $k$-prince is in the $(i,j)$-th cell of a chessboard, then it can move in the $(u,v)$-th cell if and only if
\[
 |i - u|  + |j - v| = k.
\]
Given two cells $(i,j)$ and $(u,v)$, we say that $(u,v)$ is \textit{attacked by $(i,j)$} if a $k$P can jump directly from $(i,j)$ to $(u,v)$. Note that if $(i,j)$ attacks $(u,v)$, then $(u,v)$ attacks $(i,j)$. Then, a \textit{$k$P's closed tour over an $m\times n$ board} is a sequence $(c_1,c_2,\dotsc, c_{mn}) \in \{[1,m]\times [1,n]\}^{mn}$ such that:
\begin{itemize}
 \item[(1)] $\bigcup_{i=1}^{mn} c_i= \{[1,m]\times [1,n]\}$;
\item[(2)] $c_i$ attacks $c_{i+1}$ for each $i \in [1,mn-1]$;
\item[(3)] $c_{mn}$ attacks $c_1$.
\end{itemize}
If Condition (3) is dropped, then the tour is said to be \textit{open}. In \cite{CZ}, the authors proposed the following problem:
\begin{PN}
Let $k$ be a positive integer. Can a $k$-prince make a closed tour on an $8 \times 8$ chessboard visiting each cell exactly once? 
\end{PN}
Also, they proved that the set of admissible values for $k$ is $\{1,3,5,7\}$, and  remarked that the existence of a $1$P's closed tour is a fairly easy exercise (indeed, it exists for every $2n \times 2n$ chessboard, with $n\geq 1$), while a $3$P's closed tour can be found by noticing that the standard chessboard knight is a $3$P that cannot move purely horizontally or vertically: since a closed tour of a knight exists, then a $3$P's closed tour exists. In Figure 10 of \cite{CZ} the authors directly write a $5$P's closed tour, leaving as an open problem  the existence of a $7$P's closed tour. Here, we show the $7$P's closed tour that we found by considering rotations of a smaller one:
\[
\begin{array}{|c|c|c|c|c|c|c|c|} \hline 
52 & 41 & 48 & 27 & 36 & 1 & 54 & 13 \\ \hline
43 & 46 & 25 & 34 & 7 & 30 & 51 & 56 \\ \hline
64 & 3 & 60 & 15 & 10 & 5 & 40 & 49 \\ \hline
29 & 58 & 55 & 12 & 21 & 18 & 63 & 38 \\ \hline 
6 & 31 & 50 & 53 & 44 & 23 & 26 & 61 \\ \hline
17 & 8 & 37 & 42 & 47 & 28 & 35 & 32 \\ \hline
24 & 19 & 62 & 39 & 2 & 57 & 14 & 11 \\ \hline
45 & 22 & 33 & 4 & 59 & 16 & 9 & 20 \\ \hline 
\end{array}
\]
This result, combined with the ones of \cite{CZ}, gains the following:
\begin{prop} \label{prop:all8}
There exists a $k$-prince's tour on an $8 \times 8$ chessboard if and only if $k$ is odd and not greater than $7$.
\end{prop}

We can moreover prove the following proposition:
\begin{prop} \label{prop:all2468}
For $n \in [1,4]$ there exists  a closed $k$P's tour over a $2n \times 2n$ chessboard for every odd $k \leq 2n-1$.
\end{prop}
\begin{proof}
As already remarked, there exists a closed $1$P's tour on a  $2n \times 2n$ chessboard for every integer $n\geq 1$. If $n = 2$, then we need to check the existence of a closed $3P$'s  tour over a $4 \times 4$ chessboard (we remark that a closed knight's tour on this chessboard does not exist, see \cite{W}); an example of such a tour is the following:
\[
\begin{array}{|c|c|c|c|} \hline
3 &6& 9& 14 \\ \hline
12&15&4&7\\ \hline 
5&2&13&10 \\ \hline
16&11&8&1 \\ \hline
\end{array}
\]
If $n = 3$, then we need to check the statement for $k =3,5$: since there exists a closed knight's tour then there exists a $3$-prince's closed tour, while an example for $k=5$ is the following one:
\[
\begin{array}{|c|c|c|c|c|c|} \hline
9&12&5&24&29&18 \\ \hline
20&17&34&31&26&21 \\ \hline
15&22&1&10&7&14 \\ \hline
32&25&28&19&4&33 \\ \hline
3&8&13&16&35&2 \\ \hline
36&11&6&23&30&27 \\ \hline
\end{array}
\]
The case $n=4$ follows from Proposition \ref{prop:all8}.
\end{proof}

Moreover, in \cite{K}  the author analized the closed tour problem of generalized knights over chessboards: given a positive integer $k$, the \textit{generalized $k$-knight} is the piece that is allowed to move along either the horizontal or vertical direction by $k-1$ steps, and along the other direction by $1$ step. In particular, it was proven that for every $k\geq 1$ there exists a closed tour of a generalized $k$-knight on a $2n \times 2n$ chessboard for sufficiently large values of $n$. Here, we consider Theorem 3 of \cite{K}, where we slightly modify the statement in the context of princes:
\begin{thm}
Let $n$ be a positive integer, and let $k$ be any odd integer such that $n \geq (k-1)(3k-2)$. Then, there exists a $k$P's tour over a $2n \times 2n $ chessboard.
\end{thm}
In view of the results contained in \cite{CZ} and the ones presented in this note, the following question naturally arises:
\begin{OP} 
Does there exist  a closed $k$P's tour over a $2n \times 2n$ chessboard for every odd $k \leq 2n-1$?
\end{OP}

\section*{Acknowledgements}
The author was partially supported by INdAM-GNSAGA.

\end{document}